\documentclass[12pt]{amsart}

\usepackage{fullpage}
\usepackage{setspace}
\usepackage[leqno]{amsmath}
\usepackage{amsfonts}
\usepackage{amssymb}
\usepackage{amsthm}
\usepackage{amssymb}
\usepackage{longtable}
\usepackage{bbm}

\newcommand{\la}{\lambda}

\newcommand{\C}{\mathbb{C}}
\newcommand{\N}{\mathbb{N}}

\DeclareMathOperator{\spn}{span}

\newcommand{\lrangle}[1]{\langle{#1}\rangle}

\newcommand{\lrbrack}[1]{\left[{#1}\right]}

\theoremstyle{plain}
\newtheorem{theorem}{Theorem}[section]
\newtheorem{corollary}[theorem]{Corollary}

\newtheorem{proposition}[theorem]{Proposition}

\theoremstyle{definition}
\newtheorem{definition}[theorem]{Definition}
\newtheorem{example}[theorem]{Example}

\title{Minimal nonnilpotent Leibniz algebras}

\author[Bosko-Dunbar]{Lindsey Bosko-Dunbar}
\address{Department of Mathematics, Spring Hill College\\
Mobile, AL 36608}
\email{lboskodunbar@shc.edu}

\author[Dunbar]{Jonathan D. Dunbar}
\address{Department of Mathematics, Spring Hill College\\
Mobile, AL 36608}
\email{jdunbar@shc.edu}

\author[Hird]{J.T. Hird}
\address{Department of Mathematics, West Virginia University, Institute of Technology\\
Beckley, WV 25832}
\email{John.Hird@mail.wvu.edu}

\author[Stagg]{Kristen Stagg Rovira}
\address{Department of Mathematics, California State University, Dominguez Hills\\
Carson, CA 90747}
\email{Kstagg@csudh.edu}

\begin{document}


\begin{abstract}
We classify all nonnilpotent, solvable Leibniz algebras with the property that all proper subalgebras are nilpotent. This generalizes the work of \cite{stitz_min} and \cite{proper_nilp} in Lie algebras. We show several examples which illustrate the differences between the Lie and Leibniz results.
\end{abstract}

\doublespacing
\maketitle

\section{Introduction}

Leibniz algebras were defined by Loday in \cite{loday}. They are a generalization of Lie algebras, removing the restriction that the product must be anti-commutative or that the squares of elements must be zero. One immediate consequence of this is that while the Lie algebra generated by a single element is necessarily one-dimensional, the Leibniz algebra generated by a single element (called a cyclic algebra) could be of any dimension.

Recent work in Leibniz algebra often involves studying certain classes of Leibniz algebras, such as cyclic algebras \cite{bugg-hedges}, algebras with a certain nilradical \cite{Heis,tri-nilrad}, or algebras of a certain dimension \cite{demir-4,demir-5,khudo-5}. Many of these articles involve generalizing results from Lie algebras to Leibniz algebras. Some of these results only hold over the field of complex numbers.

An algebra $L$ is called \emph{minimal nonnilpotent} if $L$ is nonnilpotent, solvable, and all proper subalgebras of $L$ are nilpotent. Minimal nonnilpotent Lie algebras were studied by Stitzinger in \cite{stitz_min}. Later Towers classified all such Lie algebras in \cite{proper_nilp}. It is the goal of this work to generalize these results to Leibniz algebras. Our results hold over any field.

\section{Results}

A Leibniz algebra $L$ is a vector space equipped with a bilinear product or bracket $ab=[a,b]$ which satisfies the Leibniz identity
$a(bc)=(ab)c+b(ac)$
for all $a,b,c \in L$.  For convenience we suppress the bracket notation for the product of individual elements 
of the algebra.  Note that we follow the notation in \cite{barnesleib,demir} and use ``left'' Leibniz algebras; some authors \cite{tri-nilrad,khudo-5} instead use ``right'' Leibniz algebras.

The following result was proven in \cite{demir} Theorem 4.16.


\begin{proposition}\label{normalizer} A Leibniz algebra $L$ is nilpotent if and only if every proper subalgebra of $L$ is properly contained in its normalizer.
\end{proposition}


\begin{definition}
Let $M$ be a subalgebra of a Leibniz algebra $L$. Define the \emph{core} of $M$ to be the maximal ideal of $L$ contained in $M$.
\end{definition}



\begin{proposition}\label{stitz_prop}
Let $L$ be a solvable Leibniz algebra and let $M$ be a self-normalizing maximal subalgebra of $L$. Let $N$ be the core of $M$. Then
\begin{enumerate}
\item $L/N$ contains a unique minimal ideal $A/N$.
\item $L/N$ is the semidirect sum of $A/N$ and $M/N$.
\item The Frattini ideal, $\phi(L/N)=0$.
\item $L/N$ is not nilpotent.
\end{enumerate}
\end{proposition}

The proof is identical to the Lie case in \cite{stitz_min} and makes use of Proposition \ref{normalizer}.

\begin{theorem}\label{min_nonnilp}
Let $L$ be a nonnilpotent, solvable Leibniz algebra all of whose proper subalgebras are nilpotent. Then $L=A \oplus \spn\{x\}$ and $A=nilrad(L)=\spn\{a_0,\ldots,a_k\} \oplus N$, with $N$ an ideal of $L$ and $x \in L$ is described by the following products:
$$xa_0=a_1, \quad xa_1=a_2, \quad \ldots, \quad xa_{k-1}=a_k, \quad xa_k=c_0a_0 + \cdots + c_ka_k$$
where $c_0 \neq 0$. Additionally $N=\langle x \rangle^2 + (\spn\{a_0,\ldots,a_k\})^2$, $A^3\leq Leib(L)$, and $p(\la)=\la^{k+1}-c_k\la^k-\cdots-c_1\la-c_0$ is irreducible. Finally, either $L$ is cyclic or $Leib(L)\leq N$.
\end{theorem}
\begin{proof}
%
$L$ contains a self-normalizing maximal subalgebra $M$, which is a Cartan subalgebra of $L$. Let $N$ be the core of $M$. By Proposition \ref{stitz_prop}, $L/N$ contains a unique minimal ideal $A/N$ which complements $M/N$ in $L/N$. So $L/A \cong M/N$ and since $M$ is nilpotent, $L/A$ is nilpotent. Since $A/N$ is nilpotent and minimal, $(A/N)^2=0$ so $A/N$ is abelian. Since $L/N$ is not nilpotent, by Engel's Theorem \cite{barnesleib,jacobsonleib}, there exists $x \in L/N$ with $x \notin A/N$ such that left-multiplication by $x$, denoted $\ell_{x}$, is not nilpotent on $L/N$. Without loss of generality, we can assume $x\in M$, $x \notin N$. Since $M/N$ is nilpotent and complements $A/N$ in $L/N$, this implies that $\ell_{x}$ restricted to $A/N$ is not nilpotent. Thus the subalgebra $B/N$ of $L/N$ generated by $A/N$ and $x$ is not nilpotent, so by the hypothesis of the theorem 
$B/N = L/N$.

Since $A/N$ is an ideal, $L = \langle x \rangle + A$. We claim that $x^2 \in N \subseteq A$. Since $M$ is nilpotent, $x^{n+1}=0$ for some $n$. Let $N_1 = \spn\{x^2,\ldots,x^n\} + N$, so that $N \leq N_1 \lneq M$.
Since $N \unlhd L$ and left-multiplication by $x^i$ is zero for $i>1$, $[N_1,L]\leq N \leq N_1$. Since $A/N$ is a minimal ideal of $L/N$, by Lemma 1.9 of \cite{barnesleib}, $[A/N,L/N]$ is 0 or anticommutative. But since $[x^i,A]=0$ for all $i>1$, $[A,x^i]$ is contained in $N$. 
From this, using the decompositions $L=\langle x \rangle + A$ and $N_1 = \spn\{x^2,\ldots,x^n\} + N$ it follows that $[L,N_1] \leq N_1$. Thus $N_1$ is an ideal of $L$ and by the maximality of $N$, $N_1=N$. Therefore $x^2 \in N$ and $L = \spn\{x\} \oplus A$.
Thus $\dim L = 1 + \dim A$, and 
$1 = \dim L/A = \dim M/N$. Define $F$ to be the one-dimensional subspace $F=\spn\{x\}$. Then we have $L=\langle x \rangle + A$ and $L = F \oplus A$, but unless $x^2=0$ the first sum is not direct and $F$ is not a subalgebra.

Let $L=M \oplus L_1$ be the Fitting decomposition of $L$ with respect to left-multiplication by $M$. Then $M/N$ is a Cartan subalgebra of $L/N$
%
and $(L_1+N)/N$ is the Fitting one-component of $L/N$ with respect to left-multiplication by $M/N$.
Since $L/N = A/N + M/N$, $L/N$ is not nilpotent and $A/N$ is a minimal ideal, we have that $M/N$ acts nontrivially and irreducibly on $A/N$.
Since $[M/N,A/N]=A/N$, the Fitting one-component of $L/N$ with respect to left-multiplication by $M/N$ is $A/N$.
Therefore $A/N=(L_1+N)/N$ and $A=L_1\oplus N$. In addition, $[N,L_1]\subseteq [M,L_1] = L_1$, 
so $[N,L_1]\subseteq N\cap L_1 = 0$.  

Let $T$ be the subalgebra of $L$ generated by $L_1$. Since left-multiplication by $M$ acts irreducibly on $L_1$ and $N$ is nilpotent, $[F,L_1]=L_1$. This implies $[F,T]=T$ and further that $[\langle x \rangle, T]=T$. Thus $\langle x \rangle + T$ is a nonnilpotent subalgebra of $L$, hence $\langle x \rangle + T = L$. Notice $x^2 \in N \leq A$ and $L_1 \leq A$ imply that $\langle x \rangle^2+T \leq A$. However $\langle x \rangle^2+T$ is a codimension 1 subalgebra of $L$, so $A = \langle x \rangle^2+T = \langle x\rangle^2 + \langle L_1 \rangle$.

Recalling that $A/N$ is abelian, we know that $A^2\leq N$, so it follows that $(L_1)^2+\lrangle{x}^2\leq N$. However, $(L_1)^2+\lrangle{x}^2$ and $N$ have the same dimension, so $(L_1)^2+\lrangle{x}^2=N$. 
Hence, $N^2 = \lrbrack{N,N}\leq Leib(L)$. Because $A=L_1\oplus N$, we know $\lrbrack{N,A}\leq Leib(L)$. By definition of $Leib(L)$, this implies $\lrbrack{A,N}\leq Leib(L)$. Thus, $A^3 = [A,A^2]\leq[A,N]\leq Leib(L)$.

Since $\ell_x|_{L_1}$ is not nilpotent, there exists an $a\in L_1$ such that $\ell_x$ is not nilpotent on $a$. 
Then $\spn\{a, xa, x(xa),\ldots, (\ell_x)^k(a)\} \subseteq L_1$, where we choose the largest $k$ such that this set is linearly independent. 
Since $M/N$ acts irreducibly on $A/N$, it follows that $F \simeq M/N$ acts irreducibly on $L_1 \simeq A/N$, 
so $\spn\{a, xa, x(xa),\ldots, (\ell_x)^k(a)\} = L_1$.
Because $L_1$ is the Fitting one-component, $(\ell_x)^{k+1}(a)=c_0 a+c_1 xa + \cdots + c_k(\ell_x)^k(a)$, and 
$c_0\ne 0$.
Note that the matrix for $\ell_x$ acting on $L_1$ is in rational canonical form, and therefore the characteristic polynomial is the minimal polynomial $p(\la)$, as given in the theorem.
If $Leib(L/N)=0$, then $Leib(L)\leq N$. Now suppose that $Leib(L/N)\ne 0$. Then there exists a minimal ideal inside of $Leib(L/N)$, and so $A/N\leq Leib(L/N)$. Since $A/N$ is a codimension 1 subalgebra of $L/N$, then $A/N=Leib(L/N)$. Thus, $Leib(L/N)$ has codimension 1 in $L/N$, which implies that $L/N$ is cyclic: $L/N = \langle \bar{z} \rangle$. Since $\langle \bar{z}\rangle$ is nonnilpotent, then $\langle z \rangle$ is nonnilpotent, and $L=\langle z \rangle$ is cyclic. 
\end{proof}

Note that the products listed in this theorem are not necessarily the only nonzero products in $L$.
However we know that $x^2 \in \langle x \rangle^2 \leq N$, $nx \in N$ for any $n \in N$, and $a_i x = - x a_i + Leib(L)$, and in the noncyclic case $Leib(L) \leq N$. Also, $A/N$ abelian means that $a_i a_j \in N$. Thus the description in the proof shows all nontrivial products in $L/N$.

Using the notation from the proof, the theorem can be restated in the following way.

\begin{corollary}\label{min_nonnilp2}
Let $L$ be a minimal nonnilpotent Leibniz algebra. Let $M$ be a self-normalizing maximal subalgebra of $L$ with core $N$, and $L=M\oplus L_1$ be the Fitting decomposition of $L$ with respect to $M$. Then $L$ is the vector space direct sum of $N$, $L_1$, and $F$ where $F$ is a one-dimensional subspace of $L$ and $M=N\oplus F$. Furthermore, $A=N\oplus L_1$ is an ideal of $L$ with $A^3\leq Leib(L)$.
\end{corollary}

In Lie algebras \cite{stitz_min,proper_nilp} prove that $A^3=0$. We recover this result for the case where $L$ is a Lie algebra and generalize to $A^3 \leq Leib(L)$ in the non-Lie case. This is due to the fact that, $A^2=N$ in Lie algebras but in Leibniz algebras we have that $A^2 \leq N$, since $N=(L_1)^2+\langle x \rangle^2=A^2+\langle x \rangle^2$. If $x^2=0$, then we would have $A^2=N$ and $A^3=0$.

%


Note that many nonnilpotent, cyclic Leibniz algebras have all proper subalgebras nilpotent (see Example \ref{cyclic1}). However there also exist nonnilpotent, cyclic Leibniz algebras with nonnilpotent subalgebras. One example is $L=\spn\{z,z^2,z^3\}$ with $zz^3=z^2+2iz^3$ over $\C$, which has a nonnilpotent subalgebra $M=\spn\{ia-a^2,a^2+ia^3\}$. Our theorem shows the structure required for minimal nonnilpotent, cyclic Leibniz algebras. For a more exhaustive study of cyclic Leibniz algebras, see \cite{batten-gin,bugg-hedges}.


\section{examples}

For the following examples we adopt the convention that when we list products of a Leibniz algebra, those not mentioned are assumed to be zero.
Note that whenever $L$ is cyclic, its generator will never be an element of either $M$ or $A$. In Example \ref{cyclic1}, neither $x$ nor $a$ is a generator, but $z=x+a$ is a generator.

\begin{example}\label{cyclic1}
Let $L$ be the cyclic Leibniz algebra $L=\spn\{z,z^2\}$ with $zz^2=z^2$. This is a minimal nonnilpotent Leibniz algebra. Then $x=z-z^2$, $a=z^2$, $N=0$, $M=\spn\{x\}$, and $A=\spn\{a\}$.
\end{example}


In Lie algebras $F=\spn\{x\}$ is a subalgebra, however in Leibniz algebras this only guaranteed to be a subspace of $L$. See Example \ref{nontrivial}. In Lie algebras, either $A$ is a minimal ideal or $A^2=Z(A)$. Either case would imply $A^3=0$, but this is clearly not the case in Example \ref{nontrivial} when $k\geq 3$.

\begin{example}\label{nontrivial}
Let $L=\spn\{x,x^2,\ldots,x^j,a,a^2,\ldots,a^k\}$ for some $j,k\in\N$ with $x^{j+1}=0$, $a^{k+1}=0$, $xa=a=-ax$, and $xa^i=ia^i$.
This is a minimal nonnilpotent Leibniz algebra. Then $N=Leib(L)=\spn\{x^2,\ldots,x^j,a^2,\ldots,a^k\}$, $F=\spn\{x\}$, $M=F\oplus N$, and $A=\spn\{a\}\oplus N$. In this example $c_0=1$ and $p(\la)=\la-1$, which is irreducible over any field. Here $A^3=\spn\{a^3,\ldots,a^k\} \neq 0$ for $k\geq 3$.
\end{example}

Over an algebraically closed field every irreducible polynomial has degree one, so the dimension of $A/N$ is one and $A=\spn\{a\}\oplus N$. Over the field of real numbers every irreducible polynomial is linear or quadratic, so either $A=\spn\{a\}\oplus N$ or $A=\spn\{a_0,a_1\}\oplus\nobreak N$. Over the rational numbers, we can construct a Leibniz algebra of this type with $A/N$ having any dimension:

\begin{example}
Over the field of rational numbers there is an irreducible polynomial of form $p(\la)=\la^{k+1}-c_k\la^k-\cdots-c_1\la-c_0$ for any $k$. Define $L=\spn\{x,a_0,a_1,\ldots,a_k\}$ with $xa_i=a_{i+1}$ for $0\leq i < k$ and $xa_k=c_0a_0 + c_1a_1 + \cdots + c_ka_k$. This is a minimal nonnilpotent Leibniz algebra. Then $N=Leib(L)=\spn\{a_1,\ldots,a_k\}$, $M=\spn\{x\}\oplus N$, $A=\spn\{a_0\}\oplus N$.
\end{example}


\begin{thebibliography}{ABCDE}











\bibitem{barnesleib} Barnes, D.  Some theorems on Leibniz algebras, \emph{Comm. Algebra}. (7) 39, (2011) 2463-2472.   




\bibitem{batten-gin} Batten-Ray, C., A. Combs, N. Gin, A. Hedges, J.T. Hird, L. Zack.  Nilpotent Lie and Leibniz algebras, \emph{Comm. Algebra}. (42) 6, (2014) 2404-2410 DOI:10.1080/00927872.2012.717655



\bibitem{Heis} Bosko-Dunbar, L., J. Dunbar, J.T. Hird, K. Stagg.  Solvable Leibniz algebras with Heisenberg nilradical, \emph{Comm. Algebra}. (43) 6, (2015) 2272-2281.

\bibitem{jacobsonleib} Bosko, L., A. Hedges, J.T. Hird, N. Schwartz, K. Stagg.  Jacobson's refinement of Engel's theorem for Leibniz algebras, \emph{Involve}. (4) 3, (2011) 293-296.

\bibitem{bugg-hedges} Bugg, K., A. Hedges, M. Lee, B. Morell, D. Scofield, S. McKay Sullivan. Cyclic Leibniz algebras. To appear, arxiv:1402.5821.





\bibitem{demir} Demir, I., K. Misra, E. Stitzinger.  On some structures of Leibniz algebras, \emph{Contemp. Math.} 623 (2014), 41-54.

\bibitem{demir-4} Demir, I., K. Misra, E. Stitzinger. On classification of four-dimensional nilpotent Leibniz algebras. \emph{Comm. Algebra}. (45) 3, (2017) 1012–1018.

\bibitem{demir-5} Demir, I. Classification of 5-dimensional complex nilpotent Leibniz algebras. To appear, arxiv:1706.00951.




\bibitem{tri-nilrad} Karimjanov, I. A.; Khudoyberdiyev, A. Kh.; Omirov, B. A. Solvable Leibniz algebras with triangular nilradicals. \emph{Linear Algebra Appl}. 466 (2015), 530–546.

\bibitem{khudo-5} Khudoyberdiyev, A. Kh.; Rakhimov, I. S.; Said Husain, Sh. K. On classification of 5-dimensional solvable Leibniz algebras. \emph{Linear Algebra Appl}. 457 (2014), 428–454.

\bibitem{loday} Loday, J.  Une version non commutative des algebres de Lie: les algebres de Leibniz, \emph{Enseign. Math.}  39 (1993) 269-293.









\bibitem{stitz_min} Stitzinger, E.  Minimal Nonnilpotent Solvable Lie Algebras, \emph{Proc. Amer. Math. Soc.} \textbf{28} 1 (1971) 47-49.

\bibitem{proper_nilp} Towers, D.  Lie Algebras All of Whose Proper Subalgebras are Nilpotent, \emph{Linear Algebra Appl.}  \textbf{32} (1980) 61-73.




\end{thebibliography}
\end{document}